\documentclass{article}
\usepackage{amsmath,amssymb,amsthm}
\usepackage{graphicx}
\usepackage{etoolbox, color, verbatim}
\theoremstyle{plain}
\newtheorem{theorem}{Theorem}[section]
\newtheorem{lemma}{Lemma}[section]
\newtheorem{corollary}{Corollary}[section]

\theoremstyle{definition}
\newtheorem{definition}{Definition}[section]

\theoremstyle{remark}

\title{Maximal dissipation and shallow water flow
-- the dam-break problem}
\author{Dalal Daw, Marko Nedeljkov, Sanja Ru\v zi\v ci\' c}
\date{}

\begin{document}

\maketitle

\begin{abstract}
The shallow water system with a bed jump describes a fluid flow
over a dam, represented by a simple step function with the 
Riemann initial data is such that the vacuum is on the right-hand side. The main question
is whether the fluid will stay on the left-hand side of the dam. After
approximating the step function space derivative with a fixed
shadow wave, solutions to that problem stay bounded.  
A proper connection between waves on both sides of the dam, 
and a unique solution afterward is found by using the maximal dissipation principle.


\end{abstract}

\section{Introduction}

The one-dimensional 
shallow water system with an uneven bed was given in \cite{KMT2017}
by using the following balance law system:
\begin{equation} \label{orig}
\begin{split}
& \partial_{t}h + \partial_{x}(uh) = 0 \text{ (mass conservation)} \\
& \partial_{t}(uh) + \partial_{x}\left( u^{2}h + g\frac{h^{2}}{2} \right)  = 
- ghb_{x}, \text{ (momentum balance)} \\
& \partial_{t}\left( h\frac{u^{2}}{2} + g\frac{h^{2}}{2} + bh \right) + 
\partial_{x}\left( guh(h + b) + h\frac{u^{3}}{2} \right)  = 0 
\text{ (energy conservation)}.
\end{split}
\end{equation}
A function $b=b(x)$ measures the rise of a bed, 
$h=h(x,t)$ measures the flow depth of the fluid, $u=u(x,t)$ is the
fluid velocity, and \(g\) is the gravity constant.
The initial data are assumed to be Riemann ones.

If the bed has a jump, the above problem is not 
solvable in the classical sense. A non-classical solution containing 
the Dirac delta function is constructed in the paper \cite{KMT2017}
by using all three equations in (\ref{orig}). Our aim is a bit different:
Find a bounded solution to the first two equations and use the energy 
relation as an admissibility condition. We believe the bounded 
solution is more realistic in flooding problems we are interested in.

For simplicity, assume that there is a single jump at $x=0$
\[
b(x)=\begin{cases}
b_{0}, & x<0 \\
b_{1}, & x>0.
\end{cases}
\]
That jump is a toy model of a dam. The natural assumption for it 
is $b_{1}>b_{0}$ without
the fluid on the right-hand side ($x>0$). 
The main question posed in this paper is to find  
conditions for the fluid not to pass over the dam. 
Mathematically, there are no non-vacuum solutions when $x>0$. 
The source function in (\ref{orig}) is 
$b_{x}=(b_{1}-b_{0})\delta(x)$, where $\delta$ stands for 
the Dirac delta function being
approximated by a zero-speed shadow wave, see (\cite{MN2010})
for that kind of solution. 
We obtain a solution by solving the mass conservation and momentum balance 
equations using the standard shadow wave computations given in that paper. 
An inequality substitutes the energy equation 
to single out a physically relevant solution.
More precisely, we will use the maximal dissipation 
admissibility condition of the energy functional 
$ h\frac{u^{2}}{2} + g\frac{h^{2}}{2} + bh$.
Dafermos propose the method in \cite{D1973}. 
One can see the papers \cite{F2014} and \cite{NR2022}  
for recent adaptations of this method.

The system analysis obtained in this way shows that there are no 
delta functions in any of the variables representing flow depth and velocity,
contrary to the results in \cite{KMT2017}. Both approximations
of $h$ and $u$ remain bounded, but the solution is not a classical
one. It can be viewed as a singular shock solution (with no point mass)
followed by a rarefaction wave. One can see such a solution
in \cite{KK1995} for the first time. 

The shadow wave approach may produce some unreasonable results. 
The maximal dissipation admissibility condition was
powerful enough to extract a unique solution here. 

Even for such simple (but still singular) cases, there are
a lot of technical problems.
The ``normal'' Riemann problem with arbitrary initial flow depth for $x>0$ 
is more complicated and not considered in this paper. 
A real challenge would be the use of 
this method in a general case of an uneven bed. We leave that problem for future work. 

The paper is organized as follows. The bed approximation is introduced
in the second section. One can find conditions for connecting waves
from the left-hand and right-hand sides of the jump there. 
That procedure resembles the one for a conservation law with 
a discontinuous flux (see  \cite{BK2008}, for example): 
The maximum dissipation admissibility condition
around the point of discontinuity at $x=0$ 
determines a proper connection between waves on both sides. 
The details are in the third section. The fourth one is devoted to 
analyzing all possible wave combinations going from the left having  
negative velocities. The final answer to the dam problem is in 
the fifth section.

\section{The bed approximation}

Take the first two equations from the system (\ref{orig}) with the
Riemann data:
\begin{equation} \label{sys}
\begin{split}
& \partial_{t}h + \partial_{x}(uh) = 0 \\
& \partial_{t}(uh) + \partial_{x}\left( u^{2}h + g\frac{h^{2}}{2} \right) = - ghb_{x}, \\
& (h,u)|_{t=0} = \begin{cases}
(h_{l},u_{l}), & x < 0 \\
(h_{r},u_{r}), & x > 0. 
\end{cases}
\end{split}
\end{equation}
The source term obtained because of the bed jump is approximated by
\[\partial_{x}b_{\varepsilon} = \begin{cases}
0, &  x \leq - \frac{\varepsilon}{2}t \\
\frac{b_{1} - b_{0}}{\varepsilon t}, &  - \frac{\varepsilon}{2}t < x \leq \frac{\varepsilon}{2}t \\
0, &  x > \frac{\varepsilon}{2}t. 
\end{cases}\]

The idea is to look for a solution in the form of a simple shadow wave
defined in \cite{MN2010}:
\begin{equation} \label{sdw}
h_{\varepsilon} = \begin{cases}
h_{0}, &  x \leq - \frac{\varepsilon}{2}t \\
x_{\varepsilon}, & - \frac{\varepsilon}{2}t < x 
\leq \frac{\varepsilon}{2}t \\
h_{1}, &  x > \frac{\varepsilon}{2}t, 
\end{cases} \quad u_{\varepsilon} = \begin{cases}
u_{0}, & x \leq - \frac{\varepsilon}{2}t \\
v_{\varepsilon}, &  - \frac{\varepsilon}{2}t 
< x \leq \frac{\varepsilon}{2}t \\
u_{1}, &  x > \frac{\varepsilon}{2}t. 
\end{cases} 
\end{equation}

A shadow wave is a piecewise smooth function depending on a
small parameter $\varepsilon \ll 1$. ``Simple'' means 
that it is represented by a piecewise constant function
with respect to $x$.
Discontinuities lie on some straight lines, as shown in (\ref{sdw}). 
The procedure for finding a shadow wave that is an approximate solution
to the given problem consists of several steps:
\medskip

\noindent - Substitute (\ref{sdw}) in (\ref{sys}).

\noindent - Calculate all expressions using standard Rankine-Hugoniot
conditions for shocks.

\noindent - Let $\varepsilon \to 0$. If equality of 
distributional limits holds, then (\ref{sdw}) is a solution.
\medskip

Note that one can use a slightly different approximation, using 
the strip $\left[-\frac{\varepsilon}{2},\frac{\varepsilon}{2}\right]$, $t>0$,
instead of the conic set $\left[-\frac{\varepsilon}{2}t,\frac{\varepsilon}{2}t\right]$, $t>0$.
Then, the term $b_{x}$ would be approximated by
\[\partial_{x}b_{\varepsilon} = \begin{cases}
0, &  x \leq - \frac{\varepsilon}{2} \\
\frac{b_{1} - b_{0}}{\varepsilon}, &  - \frac{\varepsilon}{2} < x \leq \frac{\varepsilon}{2} \\
0, &  x > \frac{\varepsilon}{2}, 
\end{cases}\]
while $x_{\varepsilon}$ is substituted by $x_{\varepsilon}t$,
and $v_{\varepsilon}$ by $v_{\varepsilon}t$ in (\ref{sdw}). 
We are using the same approximation as in \cite{MN2010} and 
the statements from it. More precisely, 
by Lemma 2.1 from the cited paper, the first equation in (\ref{sys}) 
is satisfied if and only if
\begin{equation*}
\partial_{t}h_{\varepsilon} \approx \varepsilon x_{\varepsilon}\delta, 
\text{ and } \partial_{x}\left( h_{\varepsilon}u_{\varepsilon} \right) 
\approx \left( h_{1}u_{1} - h_{0}u_{0} \right)\delta + \varepsilon 
tx_{\varepsilon}v_{\varepsilon}\delta^{'}.
\end{equation*}
And the second one, if and only if 
\begin{equation*}
\begin{split}
\partial_{t}\left( h_{\varepsilon}u_{\varepsilon} \right) \approx 
& \varepsilon x_{\varepsilon}v_{\varepsilon}\delta, \text{ and }\\
\partial_{x}\left( h_{\varepsilon}u_{\varepsilon}^{2}
+ \frac{g}{2}h_{\varepsilon}^{2} \right) \approx & \left( h_{1}u_{1}^{2} 
+ \frac{g}{2}h_{1}^{2} - h_{0}u_{0}^{2} - \frac{g}{2}h_{0}^{2} \right)\delta
+ \varepsilon t\left( x_{\varepsilon}v_{\varepsilon}^{2} 
+ \frac{g}{2}x_{\varepsilon}^{2} \right)\delta^{'}.
\end{split}
\end{equation*}
The sign ``$\approx$'' simply means a distributional limit as
$\varepsilon \to 0$. 
The source term in the second equation,
$S=-gh\partial_{x}b_{\varepsilon}$, applied to a test function 
$\varphi\in \mathcal{C}^{\infty}(\mathbb{R}\times \mathbb{R}_{+})$ gives
\[ 
\begin{split} \left\langle S,\varphi \right\rangle = & \left\langle - ghb_{x},
\varphi \right\rangle  \\ 
= & - g \int_{-\frac{\varepsilon}{2}t}^{\frac{\varepsilon}{2}t}
{x_{\varepsilon}\frac{b_{1} - b_{0}}{\varepsilon t}}\varphi(x,t)dxdt  \\ 
= & - g\frac{b_{1} - b_{0}}{\varepsilon}
\int_{0}^{\infty} \int_{- \frac{\varepsilon}{2}t}^{\frac{\varepsilon}{2}t}
\frac{x_{\varepsilon}}{t} 
\left( \varphi(0,t) + x\partial_{x}\varphi(0,t) 
+ \mathcal{O}\left( x^{2} \right) \right)dx dt  \\ 
\approx & \int_{0}^{\infty}{- g\frac{b_{1} - b_{0}}{\varepsilon}
\frac{x_{\varepsilon}}{t}}\varepsilon t\varphi(0,t)dt
= - \int_{0}^{\infty}g(b_{1} - b_{0})x_{\varepsilon}\varphi(0,t)dt. 
\end{split} \]
That is
\[S \approx - g\left( b_{1} - b_{0} \right)x_{\varepsilon}\delta.\]
The delta function and its 
derivative are supported by the line \(x = 0\).
There is no way to cancel  the \(\delta^{'}\) term 
unless \( \lim_{\varepsilon \to 0} \varepsilon {x_{\varepsilon} 
= 0}\), and  
\( \lim_{\varepsilon \to 0} v_{\varepsilon} = 0.\)
Particularly, it means that there are no delta functions in the
limit of the approximation $(h_{\varepsilon},u_{\varepsilon})$. 
With \(x_{\varepsilon} = \chi \geq 0\), $v_{\varepsilon}=0$,
and taking the distributional limit, the system (\ref{sys}) 
is satisfied under two conditions:

\noindent
-- The first one is \(h_{1}u_{1} = h_{0}u_{0}\). Note that it implies 
\( \mathop{\rm sign}{(u}_{0}) 
= \mathop{\rm sign}{(u}_{1})\).

\noindent
-- The second one is 
\(h_{1}u_{1}^{2} + \frac{g}{2}h_{1}^{2} - h_{0}u_{0}^{2} - \frac{g}{2}h_{0}^{2} 
= - g{(b}_{1} - b_{0})\chi.\)
An arbitrary constant \(\chi \) is non-negative due 
to the obvious physical reasons since a negative fluid depth is meaningless.

From the first relation, we have
\begin{equation} \label{u1}
u_{1} = u_{0}\dfrac{h_{0}}{h_{1}}.
\end{equation}
Substituting such \(u_{1}\) in the second one, gives
\[u_{0}^{2}\frac{h_{0}^{2}}{h_{1}} + \frac{g}{2}h_{1}^{2} - h_{0}u_{0}^{2} - \frac{g}{2}h_{0}^{2} = - g{(b}_{1} - b_{0})\chi. \]
Multiplication by \(h_{1}\) transfers it to the cubic equations
with respect to $h_{1}$, 
\[\frac{g}{2}h_{1}\left( h_{1}^{2} - h_{0}^{2} \right) 
- u_{0}^{2}{h_{0}(h}_{1} - h_{0}) + g{(b}_{1} - b_{0})\chi h_{1} = 0\]
or equivalently, 
\begin{equation} \label{h1}
h_{1}^{3} - \left( h_{0}^{2} + \frac{2}{g}h_{0}u_{0}^{2} - 2(b_{1}-b_{0}) 
\chi  \right) h_{1} + \frac{2 h_{0}^{2}u_{0}^{2}}{g}  = 0.
\end{equation}

Note that the cubic equation is of the form 
\( h_{1}^{3} - Ah_{1} + B = 0 \). A simpler case when  \( A,\ B \geq 0\)
occurs only if \(b_{1} < b_{0}\), and that case  is not interesting here. 
Otherwise, a sign of \(A\) depends on particular initial
data. A short analysis of the above cubic equation follows.
Put 
\[
C(x):=x^{3}-\left( h_{0}^{2}+\frac{2}{g}h_{0}u_{0}^{2}-2(b_{1}-b_{0})\chi\right)x
+\frac{2u_{0}^{2}h_{0}^{2}}{g}=0,
\]
and look for a positive root of this equation. 
Denote $A:=h_{0}^{2}+\frac{2}{g}h_{0}u_{0}^{2}-2(b_{1}-b_{0})\chi$.
The function $C$ is positive for $\chi$
large enough, and there are no non-negative roots since
$C(0)=\frac{2u_{0}^{2}h_{0}^{2}}{g}>0$ and $C(\infty)=\infty$. 
For smaller values of $\chi$, $C$ attains its positive minimum at the point
\[
x_{e}=\sqrt{\frac{A}{3}},\; \text{ provided that }\chi
\leq \frac{h_{0}^{2}+\frac{2}{g}h_{0}u_{0}^{2}}
{2(b_{1}-b_{0})}.
\]
So, a positive solution will exist if 
$C(x_{e})=\dfrac{A}{3}\sqrt{\dfrac{A}{3}}-A\sqrt{\dfrac{A}{3}}
+\dfrac{2u_{0}^{2}h_{0}^{2}}{g} \leq 0$, i.e.\ if $A^{3}\geq 
\dfrac{27}{g^{2}}u_{0}^{4}h_{0}^{4}$. That gives the following conditions
\begin{equation} \label{chi}
\chi \leq  \frac{h_{0}^{2}+\frac{2}{g}h_{0}u_{0}^{2}-\frac{3}{g^{2/3}}
u_{0}^{4/3}h_{0}^{4/3}}{2(b_{1}-b_{0})}=:\overline{\chi}
\end{equation}
and
\begin{equation} \label{0}
h_{0}+\frac{2}{g}u_{0}^{2}  \geq  \frac{3}{g^{2/3}}
u_{0}^{4/3}h_{0}^{1/3},
\end{equation}
since the flow depth must be non-negative, i.e., $\chi \geq 0$.
Using $y:=\frac{u_0^2}{gh_0}>0$, (\ref{0}) reads as
\[1+2y\geq 3\sqrt[3]{y^2}.\]
Put $\tilde{f}(y):=1+2y-3\sqrt[3]{y^2}.$ Then, $\tilde{f}(0)=1>0$, $f(1)=0$ and 
\[\tilde{f}'(y)=2\Big(1-\frac{1}{\sqrt[3]{y}}\Big)\begin{cases}
<0, & \text{ if } y<1\\
>0 & \text{ if } y>1.
\end{cases}\]
So, $\min_{y\geq 0} \tilde{f}(y)=f(1)=0$ and (\ref{0}) is always satisfied.

For a solution to the initial data problem (\ref{sys}), one has to find proper
\( (h_{1},u_{1}) \) and \( (h_{0}, u_{0}) \) such that (\ref{u1}) holds  for some \(\chi \geq 0\). 
The next step would be to connect states (\(h_{l}{,u}_{l}\)), 
\((h_{0},u_{0})\) by waves with negative velocities, 
and \( (h_{1},u_{1}) \), \( (h_{r},u_{r})\)
by ones with positive velocities. These waves are
shocks and rarefaction solutions to the shallow water system with an even bed
\begin{equation} \label{even}
\begin{split}   & \partial_{t}h + \partial_{x}(uh) =  0 \\
& \partial_{t}(uh) + \partial_{x}\left( u^{2}h + g\frac{h^{2}}{2} \right) =  0
\end{split}
\end{equation}

Let us find conditions for the solution to satisfy the entropy condition 
around the bed jump ($x=0$). 

\section{Admissibility conditions around the bed jump}

The energy balance relation will be used to find a way to connect
classical solutions from both sides of $x=0$. That will reveal 
necessary conditions which 
\(h_{0}{,h_{1},u}_{0},u_{1},\ x_{\varepsilon},\ v_{\varepsilon}\)
have to satisfy. 

The energy balance in the third equation in (\ref{sys}) reads as
\[
\begin{split}
& \partial_{t} \eta(h,u) + \partial_{x} Q(h,u) \\
= & \partial_{t}\left( h\frac{u^{2}}{2} + g\frac{h^{2}}{2} + bh \right)
+ \partial_{x}\left( guh(h + b) + h\frac{u^{3}}{2} \right) = 0,
\end{split}
\]
where $\eta$ denotes the energy, and $Q$ its flux. 
The solution around a bed jump has to satisfy (in the distributional
sense) the Lax entropy condition
\[
\partial_{t}\left( h_{\varepsilon}\frac{u_{\varepsilon}^{2}}{2} 
+ g\frac{h_{\varepsilon}^{2}}{2} + b_{\varepsilon}h_{\varepsilon} \right) 
+ \partial_{x}\left( gu_{\varepsilon}h_{\varepsilon}\left( h_{\varepsilon} 
+ b_{\varepsilon} \right) + h_{\varepsilon}\frac{u_{\varepsilon}^{3}}{2} \right) 
\leq 0,  \text{ as } \varepsilon \rightarrow 0.
\]

Using Lemma 10.1 from \cite{MN2010}, and denoting
\( \beta_{\varepsilon} = \dfrac{b_{1} - b_{0}}{\varepsilon t}\ \), we
obtain
\[\begin{split}
\partial_{t}\left( h_{\varepsilon}\frac{u_{\varepsilon}^{2}}{2} 
+ g\frac{h_{\varepsilon}^{2}}{2} + b_{\varepsilon}h_{\varepsilon} \right) 
& \approx \frac{d}{dt}\Big(\varepsilon t\left( x_{\varepsilon}\frac{v_{\varepsilon}^{2}}{2} 
+ g\frac{x_{\varepsilon}^{2}}{2} + \beta_{\varepsilon}x_{\varepsilon} \right)\Big)
\delta\\
& \approx \frac{d}{dt}\big((b_1-b_0)x_\varepsilon\big)\delta 
\approx 0
\end{split}
\]
and 
\[ 
\begin{split} 
& \partial_{x}\left( gu_{\varepsilon}h_{\varepsilon}
\left( h_{\varepsilon} + b_{\varepsilon} \right) 
+ h_{\varepsilon}\frac{u_{\varepsilon}^{3}}{2} \right) \\
\approx & \left( g\left( u_{1}h_{1}\left( h_{1} + b_{1} \right) 
- u_{0}h_{0}\left( h_{0} + b_{0} \right) \right) \right)\delta 
+ \varepsilon t(gv_{\varepsilon}x_{\varepsilon}\left( x_{\varepsilon} 
+ \beta_{\varepsilon} \right) 
+ x_{\varepsilon}\frac{v_{\varepsilon}^{3}}{2})\delta^{'}
\end{split} 
\]
Since \(\delta^{'}\) changes its sign, the term
\(\varepsilon t\left(gv_{\varepsilon}x_{\varepsilon}
\left( x_{\varepsilon} + \beta_{\varepsilon} \right) 
+ x_{\varepsilon}\frac{v_{\varepsilon}^{3}}{2}\right)\)
has to converge to zero as \(\varepsilon \rightarrow 0\). More
precisely,\(\ v_{\varepsilon} = 0\) or \(x_{\varepsilon} = 0\) because
\(\varepsilon t\beta_{\varepsilon} \rightarrow b_{1} - b_{0} \neq 0\).
The condition \(x_{\varepsilon} \approx 0 \) would mean that there is no water around the origin, 
while \(v_{\varepsilon} = 0\) is more 
physically reasonable 
-- water stops moving at the exact bed jump place.
Note that the second condition is also necessary for the existence
of shadow wave solution as seen in the previous section, and let us use 
the notation from there.
Put  \(x_{\varepsilon} = \chi \geq 0\) and \(v_{\varepsilon} = 0\) 
in the rest of the paper. 
Then, the entropy inequality reads as
\[ g\left( u_{1}h_{1}\left( h_{1} + b_{1} \right) 
- u_{0}h_{0}\left( h_{0} + b_{0} \right) \right) \leq 0.\]
With \(\ h_{1}u_{1} = h_{0}u_{0}\), the entropy condition reduces to
\begin{equation} \label{ent}
h_{1} + b_{1} \leq h_{0} + b_{0}, \quad \text{ i.e. } h_{1} \leq h_{0} - \left( b_{1} - b_{0} \right).
\end{equation}
That is, a fluid depth on the right-hand side of the dam is less
than the difference between the fluid depth on the left-hand side
and the dam's size, which appears to be physically realistic.
This condition is independent of the specific choice of $\chi$
that satisfies (\ref{chi}), meaning it does
not provide enough precision to extract a unique solution.
To address this, we will apply Dafermos'  maximal dissipation  
(see \cite{D1973}, \cite{D2012}, or \cite{F2014})
to obtain uniqueness. This approach was successfully used
for the shadow wave solutions in \cite{NR2022},
and we will adopt the results and notation from that paper.

The total energy of a solution $U=(h,u)$ in the interval $[-L,L]$ 
at time $t>0$ is given by 
\[ H_{[-L,L]}(U(\cdot,t)):=\int_{-L}^{L}\eta(U(x,t))dx,\]
A value $L>0$ is large enough to ensure that $U$ is constant
out of $[-L,L]$.

\begin{definition}(\cite{D1973}) \label{d_optim}
A weak solution $U=(h,u)$ to the system (\ref{sys})
satisfies the energy admissibility condition on $\mathbb{R}\times (0,T]$
if there is no solution $\bar{U}$ such that for some $\tau\in [0,T)$ we have
$U=\bar{U}$, $t<\tau$ and $\frac{d}{dt} 
H_{[-L,L]}(\bar{U}(\cdot,\tau))
<\frac{d}{dt}H_{[-L,L]}(U(\cdot,\tau))$. 
The derivatives are assumed to be forward in time.
\end{definition}

The energy production for a shock wave 
\[ 
U(x,t)=\begin{cases}U_{0}(x,t), & x<c(t) \\
U_{1}(x,t), & x>c(t)
\end{cases}
\]
is defined by
\[
\begin{split}
\frac{d}{dt}H_{[-L,L]}(t) = & -c'(t)
\big( \eta(U_{1}(c(t),t)) - \eta(U_{0}(c(t),t))\big) \\
& +\big( Q(U_{1}(c(t),t)) - Q(U_{0}(c(t),t)) \big) 
+Q(U_{0}(-L,t)) - Q(U_{1}(L,t)) \\
=: & \mathcal{D}(t)+Q(U_{0}(-L,t)) - Q(U_{1}(L,t)). 
\end{split}
\]
Note that the term $Q(U_{0}(-L,t)) - Q(U_{1}(L,t))$ is fixed due to
the initial data. So, only the local energy production 
denoted by $\mathcal{D}(t)$ 
is important for Definition \ref{d_optim}. 

The energy production for a rarefaction wave is 
$\frac{d}{dt}H_{[-L,L]}(t)=Q(U_{0}(-L,t)) - Q(U_{1}(L,t))$
(see \cite{D2012}), i.e.\ $\mathcal{D}(t)$ equals zero due to a continuity 
of the rarefaction wave.

\begin{definition}
Suppose that there exists a solution to the Riemann problem
with left and right initial states $U_l=(h_l,u_l)$ and $U_1=(h_1,u_1)$,
respectively and that it consists of elementary waves with negative speed.
Also, suppose that there exists a solution to the Riemann problem with
left and right side initial data $U_1$ and $U_r$, 
consisting of waves with positive speed. Then a  pair $(U_0,U_1)$ 
is called connection.
\end{definition}
The connection appears in equations with discontinuous jump (\cite{BK2008}).

\begin{lemma} \label{lemma1}
For a given $(h_{0},u_{0})$ on the left-hand side of the bed jump,
satisfying $h_{0}>b_{1}-b_{0}>0$, $u_{0}>0$ and
\begin{equation}\label{exist}
[b]=b_1-b_0\leq \frac{h_0}{2}(3-\sqrt{1+8y}), \quad y:=\frac{u_0^2}{gh_0}\leq 1,
\end{equation}
there exists a unique choice
up to $\varepsilon$ of the right-hand values $(h_{1},u_{1})$ 
that minimizes the energy
production of the static shadow wave at $x=0$  and satisfies the entropy condition (\ref{ent}).
The right-hand values are 
\begin{equation}\label{opth1}
h_1=\begin{cases}
\sqrt[3]{\frac{h_{0}^{2}u_{0}^{2}}{g}}, & \text{ if }
\sqrt[3]{\frac{h_{0}^{2}u_{0}^{2}}{g}}\leq h_0-(b_1-b_0)\\
h_0-(b_1-b_0), & \text{ otherwise }
\end{cases},\quad u_{1}=\frac{h_{0}u_{0}}{h_{1}}.
\end{equation}

\end{lemma}

\begin{proof}
Assume that $(h_{0},u_{0})$ and $(h_{1},u_{1})$ are known.
Applying the above procedure from \cite{NR2022} 
to (\ref{sys}) and its solution (\ref{sdw})
around the origin with
\[ 
\eta=h\frac{u^{2}}{2}+\frac{g}{2}h^{2}+bh, \text{ and } 
Q=ghu(h+b)+h\frac{u^{3}}{2}
\]
one gets
\[
\begin{split}
\mathcal{D}(t) = & \frac{\varepsilon}{2}\left(\eta(\chi,0) 
- \eta(h_{0},u_{0})\right)+Q(\chi,0)-Q(h_{0},u_{0}) \\
& - \frac{\varepsilon}{2}\left( \eta(h_{1},u_{1}) 
-\eta(\chi,0) \right) +Q(h_{1},u_{1})-Q(\chi,0) \\
= & Q(h_{1},u_{1}) - Q(h_{0},u_{0}) 
-\frac{\varepsilon}{2}\left(\eta(h_{0},u_{0}) 
-\eta(h_{1},u_{1})\right) \\
\approx & Q(h_{1},u_{1}) - Q(h_{0},u_{0}) \\
= & gh_{1}u_{1}(h_{1}+b_{1})+h_{1}\frac{u_{1}^{3}}{2}
- gh_{0}u_{0}(h_{0}+b_{0}) - h_{0}\frac{u_{0}^{3}}{2}. 
\end{split}
\]
Now, we have to find a positive $\chi$ satisfying (\ref{chi}) such that 
$\mathcal{D}(t)$ is minimal. Denote by $H_{0}$ a part of $\mathcal{D}$
depending on $\chi$,
\[
H_{0}=gh_{1}u_{1}(h_{1}+b_{1})+h_{1}\frac{u_{1}^{3}}{2}
=gh_{0}u_{0}(h_{1}+b_{1})+\frac{h_{0}^{3}u_{0}^{3}}{2h_{1}^{2}},
\] 
where we have used (\ref{u1}). Then 
\[\frac{\partial H_0}{\partial h_1}=gh_{0}u_{0}
-\frac{h_{0}^{3}u_{0}^{3}}{h_{1}^{3}}.\]
$H_0$ has a stationary point at $\bar{h}_1$, where
\[
\overline{h}_{1}=\sqrt[3]{\frac{h_{0}^{2}u_{0}^{2}}{g}}.
\]
So, if $u_0>0$, the function $H_0$ is decreasing for
$h_1<\bar{h}_1$ and increasing for $h_1>\bar{h}_1$.
Thus, the minimum of $H_0$ is reached at $h_1=\bar{h}_1$. 
Moreover, from (\ref{h1}) we have
\[\chi(\bar{h}_1)=\bar{\chi}=\frac{h_0^2+\frac{2}{g}h_0u_0^2
-\frac{3}{\sqrt[3]{g^2}}\sqrt[3]{u_0^4h_0^2}}{2[b]}
=\frac{h_0^2}{2[b]}(1+2y-3\sqrt[3]{y}), \quad y=\frac{u_0^2}{gh_0}.\]
Due to (\ref{0}), $\bar{\chi}$ is positive. If $\bar{h}_1\leq h_0-[b],$
 the corresponding SDW is entropic. 

If $\bar{h}_1>h_0-[b],$ then $H_0$ has the minimum which satisfies 
the entropy condition at some $\tilde{h}_1\leq h_0-[b]$. 
We have to check if the sign of the value $\chi(\tilde{h}_1)$ 
defined by (\ref{h1}) is non-negative. We have,
\[\chi(h_1)=\frac{1}{2[b]}\Big((1+2y)h_0^2-h_1^2-2y\frac{h_0^3}{h_1}\Big)\geq 0\]
if $2yh_0^2\frac{h_0-h_1}{h_1}\leq h_0^2-h_1^2$.
The above inequality reduces to
\begin{equation}\label{inq}
(h_0-h_1)(h_1^2+h_1h_0-2yh_0^2h_1)\geq 0.
\end{equation}
The function $k(h_1):=h_1^2+h_1h_0-2yh_0^2h_1$ is negative for 
\[h_1\in\big(\frac{h_0}{2}(-\sqrt{1+8y}-1),\frac{h_0}{2}(\sqrt{1+8y}-1)\big),\]
and non-negative otherwise. So, if $y\leq1$ the inequality (\ref{inq}) holds for 
\[h_1\in\big[\frac{h_0}{2}(\sqrt{1+8y}-1), h_0\big].\]
Note that it is not necessary to consider $h_1>h_0$ since, in that case, 
the entropy condition does not hold.
If $y>1$, (\ref{inq}) does not hold for any $h_1\geq 0$ that satisfies the entropy condition.
Thus,  $\tilde{h}_1$ exists if $y\leq 1$ and 
\[\frac{h_0}{2}(\sqrt{1+8y}-1)\leq h_0-[b],\]
i.e.\
\[[b]\leq \frac{h_0}{2}(3-\sqrt{1+8y}).\]   

Otherwise, an entropic SDW at the origin does not exist, i.e.,
there is no entropic connection for a given data
Note that $\frac{d\chi}{dh_1}>0$ for any $h_1<\bar{h}_1$. That is,
if $\chi(\tilde{h}_1)<0$, then $\chi(h_1)<0$ for $h_1<\tilde{h}_1$, too.
Thus, if $\bar{h}_1>h_0-[b]$, the only possible minimum point is
exactly $\tilde{h}_1=h_0-[b]$.
\end{proof}
\medskip

That concludes the analysis around $x=0$, where a bed has a jump. 
The next step is to find $ h_0 $ such that the total entropy production is minimal 
and the connection satisfies the entropy condition. 
If it is not possible to find $h_0$ such that entropy condition holds for the choice 
$h_1=\sqrt[3]{\dfrac{h_{0}^{2}u_{0}^{2}}{g}}$, 
one has to find for which $h_0$ and $h_1$ taken such that $h_1=h_0-[b]$ 
the solution exist. Furthermore, if an admissible solution exists for both choices, 
the relation between the obtained optimal values of total entropy production
will determine the proper choice for $ h_0$.
Now, let us look at the classical waves that are away from the jump.

\subsection{ Elementary waves for homogenous system}

Şystem (\ref{sys}) is the same as the isentropic gas model with
\(h\ \) playing a role of density,
\[
\begin{split} 
& \partial_{t}h + \partial_{x}(uh) = 0\\
& \partial_{t}(uh) + \partial_{x}\left( u^{2}h + p(h) \right) = 
- g\partial_{x}bu,\; p(h) = \ g\frac{h^{2}}{2}.
\end{split} \]

The characteristics of that system are 
\[\lambda_{1,2} = u \pm \sqrt{p^{'}(h)}, \text{ that is } 
\lambda_{1} = u - \sqrt{gh}, \text{ and } \lambda_{2} = u + \sqrt{gh}. \]

The shock curves are given by
\[ 
\begin{split} 
S_{1}:\; &  u = u_{l} - \sqrt{\frac{g}{2h_{l}h}\left( h - h_{l} \right)
\left( h^{2} - h_{l}^{2} \right)} \\ & 
= u_{l} - \sqrt{\frac{g\left( h_{l} + h \right)}{2h_{l}h}} \left( h - h_{l} 
\right)= u_{l} - \sqrt{\frac{g}{2}}.\left( h - h_{l} \right)
\sqrt{\frac{1}{h} + \frac{1}{h_{l}}}, \; h > h_{l} \\
S_{2}:\; &  u = u_{l} - \sqrt{\frac{g}{2h_{l}h}\left( h - h_{l} \right)
\left( h^{2} - h_{l}^{2} \right)} \\ = & u_{l} - \sqrt{\frac{g}{2}} \left( h_{l} 
- h \right)\sqrt{\frac{1}{h} + \frac{1}{h_{l}}}, \; h < h_{l}.
\end{split}
\]
Wave speeds for the shocks are given by
\[
c = u_{l} - \sqrt{\frac{g}{2}
\left( \frac{1}{h} + \frac{1}{h_{l}} \right)} h 
= u_{l} - \sqrt{\frac{g\left( h + h_{l} \right)h}{2h_{l}}},
\text{ for } S_{1}
\]
and
\[
c = u_{l} + \sqrt{\frac{g\left( h + h_{l} \right)h}{2h_{l}}}, \text{ for } S_{2}.
\]

Assume that \(u_{l} > 0\), i.e., the water runs from left to right.
One can treat the opposite case analogously, but that is not important for
the dam problem.

The rarefaction curves are 
\[ 
\begin{split} 
R_{1}:\;   hu =  & hu_{l} + 2h\left( \sqrt{gh_{l}} - \sqrt{gh} \right),\;
h<h_{l} \\
R_{2}:\;  hu = & hu_{l} - 2h\left( \sqrt{gh_{l}} - \sqrt{gh} \right),\; 
h>h_{l} 
\end{split} 
\]
or
\[ 
\begin{split} 
R_{1}:\; u = & u_{l} + \sqrt{gh_{l}} - \sqrt{gh},\; h<h_{l} \\
R_{2}:\; u = & u_{l} - \sqrt{gh_{l}} + \sqrt{gh},\; h>h_{l} 
\end{split} 
\]
A rarefaction wave connecting states $(h_{l},u_{l})$ and 
$(h_0,u_0)$ is given by 
\[
(h,u)(x,t) = \begin{cases}
(h_{l},u_{l}), & \frac{x}{t} \leq \lambda_{i}{(h}_{l},u_{l}) \\
R_{i}(s), & \lambda_{i}{(h}_{l},u_{l}) \leq \frac{x}{t} = \lambda_{i}(s)
\leq \lambda_{i}{(h}_{0},u_{0}) \\
(h_{0},u_{0}), & \frac{x}{t} > \lambda_{i}{(h}_{0},u_{0}),

\end{cases} ,\quad i=1,2.
\]
There is an implicit relation for $s$ above, but
the exact formulas are not important 
for the analysis of solutions.

\section{Left-hand sided waves}

Proving the existence of solutions on the left-hand side 
means showing that it is 
possible to find appropriate $h_{0}>0$ and $u_{0}$ such that 
system (\ref{even}) with the Riemann initial data
\[
(h,u)|_{t=0} = \begin{cases}
(h_{l},u_{l}), & x < 0 \\
(h_{r},u_{r}), & x > 0 
\end{cases}
\]
has a solution consisting of waves with negative speed.

One could see that for \(\ S_{2}\), the wave speed is positive. That
is that there are no such waves on the left-hand side.

An \(R_{2}\) propagates with different speeds, but the maximal one equals
\(\lambda_{2}{(h}_{0},u_{0})\). It has to be negative, so
\begin{equation} \label{r2r}
\lambda_{2}{(h}_{0},u_{0}) = u_{0} + \sqrt{gh_{0}} < 0, 
\text{ and  } u_{0} < - \sqrt{gh_{0}}<0.
\end{equation}
%Specially, \(u_{0} < 0\).

Thus, possible solutions on the left-hand side can be
\(S_{1},\ R_{1},\ R_{2},\ S_{1} + R_{2}, \text{ or } R_{1} + R_{2}\)

For the $S_1$, starting at
\({(h}_{l},u_{l}),\) one has to find the intermediate state
\((h_{\ast},u_{\ast}) = S_{1}(h_{l},u_{l})\),
and then the second wave  at 
\((h_{\ast},u_{\ast}),$ such that  $(h_{0},u_{0}) = R_{2}(h_{\ast},u_{\ast})\), provided 
\(\lambda_{2}(h_{\ast},u_{\ast}) \leq \lambda_{2}{(h}_{0},u_{0}) < 0\).
If \( (h_{0},u_{0}) = (h_{\ast},u_{\ast}) \),  a solution is a single shock wave.

With \(R_{1}\) instead of \(S_{1}\), one has to
find  \({(h}_{\ast},u_{\ast})\) such that
\((h_{\ast},u_{\ast}) = R_{1}(h_{l},u_{l})\), with
\((h_{0},u_{0}) = R_{2}(h_{0},u_{0})\). If \( (h_{0},u_{0}) 
= (h_{\ast},u_{\ast}) \), a solution is an $R_{1}$.

The last possibility is a single $R_{2}$ on the left-hand side. Then 
\( (h_{0},u_{0}) = R_{2}(h_{l},u_{l})\)

For the shock wave, the necessary condition is the negativity of its speed,
\[
u_{l} - \sqrt{\frac{g\left( h_{\ast} + h_{l} \right)h_{\ast}}{2h_{l}}} < 0, 
\text{ that is, }
gh_{\ast}^{2} + gh_{l}h_{\ast} - 2h_{l}u_{l}^{2} > 0.
\]

For the rarefaction wave, the following conditions should hold
\[\lambda_{2}(h_{\ast},u_{\ast}) \leq u_{\ast} + \sqrt{gh_{\ast}} \leq u_{0} + \sqrt{gh_{0}} < 0\]
and it implies relation $u_{0}<-\sqrt{gh_{0}}$ that is already obtained
above. 

\begin{lemma} \label{lemma2}
Suppose that $u_{l}\geq 0$. A solution for $x<0$ may consist of
the following waves and their combinations.
\begin{enumerate}
\item A single $S_{1}$ when
\begin{equation} \label{SR2}
h_{0} > \max \left\{ h_{l}, \; \frac{h_l}{2}\left( \sqrt{1 + 8z} - 1 \right)\right\},
\quad z:=\frac{u_l^2}{gh_l}
\end{equation}
\begin{equation} \label{S}
u_{0} = u_{l} - \sqrt{\frac{g}{2} \left( \frac{1}{h_{0}} + \frac{1}{h_{l}} \right)}
\left( h_{0} - h_{l} \right), \; h_{0} > h_{l}.
\end{equation}
\item A combination $S_{1}+R_{2}$, where $u_{0}$ satisfies
\begin{equation} \label{SR}
u_{l} - \sqrt{\frac{g}{2}\left( \frac{1}{h_{l}} + \frac{1}{h_{0}} \right)}
\left( h_{0} - h_{l} \right) < u_{0} 
< -\sqrt{gh_{0},}\;  h_{0} > h_{l},
\end{equation}
and $h_{0}$ satisfies (\ref{SR2}).
\end{enumerate}
\end{lemma}

\begin{proof}
Let us now check all combinations of elementary waves.

\noindent \textbf{1.}
In the case $R_{1} + R_{2}, $ we have to solve the following   system
\begin{equation*}
\begin{split}
& u_{\ast} = u_{l} + \sqrt{gh_{l}} - \sqrt{gh_{\ast}},\; h_{\ast}< h_{l} \\
& u_{0} = u_{\ast} - \sqrt{gh_{\ast}} + \sqrt{gh_{0}},\; h_{0}>h_{\ast}
\end{split}
\end{equation*}
with respect to the variables \(\ (h_{\ast},u_{\ast})\). Substitution of the first in the
second equation gives
\[u_{0} = u_{l} + \sqrt{gh_{l}} - 2\sqrt{gh_{\ast}} + \sqrt{gh_{0}},
\text{ where } h_{\ast}\leq \min \{h_{l},h_{0}\}\]
and that implies \( u_{0} \leq u_{l} + \sqrt{g}
\left( \sqrt{h_{0}} + \sqrt{h_{l}} \right)\). 
The condition on \({h}_{\ast}\) implies
\[u_{0} \geq u_{l} + \sqrt{gh_{0}} - \sqrt{gh_{l}},\text{ and } 
u_{0} \geq u_{l} + \sqrt{gh_{l}} - \sqrt{gh_{0}}.\]
These two bounds, taken together, may be written as
\[u_{0} - u_{l} > \left| \sqrt{g}\left( \sqrt{h_{0}} 
- \sqrt{h_{l}} \right) \right|.\]
These bounds give \(u_{0} > u_{l}\) that contradicts (\ref{r2r}),
since \(u_{l} \geq 0\). Thus, it is impossible to have a combination of 
two rarefaction waves on the left-hand side.

\noindent \textbf{2.}
To get \(S_{1} + R_{2}\ \ \)solution, one has to find
\((h_{\ast},u_{\ast})\) such that
\[ 
u_{\ast} = u_{l} - \sqrt{\frac{g}{2h_{l}h_{\ast}}\left( h_{\ast} - h_{l} \right)
\left( h_{\ast}^{2} - h_{l}^{2} \right)}, \text{ with }
u_{l} - \sqrt{\frac{g\left( h_{\ast} + h_{l} \right)h_{\ast}}{h_{l}}}<0, \; 
h_{\ast} > h_{l},
\]
and 
\[
u_{0} = u_{\ast} - \sqrt{gh_{\ast}} + \sqrt{gh_{0}}, \text{ with } 
u_{0} + \sqrt{gh_{0}} < 0, \; h_{0} > h_{\ast}.
\]
Again, substituting the first equation in the second one,
\[
\begin{split}
u_{0} = & u_{l} - \sqrt{\frac{g\left( h_{\ast} + h_{l} \right)}{{2h}_{l}h_{\ast}}}\left( h_{\ast} 
- h_{l} \right) + \sqrt{gh_{0}} - \sqrt{gh_{\ast}}\\
= & u_{l} - \sqrt{\frac{g}{2}\left( \frac{1}{h_{l}} + \frac{1}{h_{\ast}} \right)}\left( h_{\ast} 
- h_{l} \right) + \sqrt{gh_{0}} - \sqrt{gh_{\ast}},
\end{split}
\]
where \(h_{l} < h_{\ast} <  h_{0} \). That implies \(h_{0} > h_{l}\). Then, 
assuming that $u_{0}=u_{0}(h_{\ast})$, 
\[
\begin{split}
\frac{{du}_{0}}{{dh}_{\ast}} = & \frac{\frac{g}{2} \frac{1}{h_{\ast}^{2}}}{2\sqrt{\frac{g}{2}
\left( \frac{1}{h_{l}} + \frac{1}{h_{\ast}} \right)}}\left( h_{\ast} - h_{l} \right) 
- \sqrt{\frac{g}{2}\left( \frac{1}{h_{l}} + \frac{1}{h_{\ast}} \right)} 
- \sqrt{\frac{g}{2}.\frac{1}{h_{\ast}}} \\ 
= & \frac{\frac{g}{2}.\frac{1}{h_{\ast}^{2}}}{2\sqrt{\frac{g}{2}
\left( \frac{1}{h_{l}} + \frac{1}{h_{\ast}} \right)}}\left( \frac{1}{h_{\ast}} 
- \frac{h_{l}}{h_{\ast}^{2}} - \frac{1}{h_{l}} - \frac{1}{h_{\ast}} \right) 
- \sqrt{\frac{1}{h_{\ast}}\left( \frac{1}{h_{l}} + \frac{1}{h_{\ast}} \right)} < 0.
\end{split}
\]
So, \(u_{0}\left( h_{\ast} \right)\) is a decreasing function with respect to
\(h_{\ast}\). Then
\(u_{0}(h_{0}) < \ u_{0}(h_{\ast}) 
< u_{0}(h_{l}).\)
The value $u_{0}$ is already bounded from above by $-\sqrt{gh_{0}}<0$.
Writing these two bounds together,
\begin{equation*} 
u_{l} - \sqrt{\frac{g}{2}\left( \frac{1}{h_{l}} + \frac{1}{h_{0}} \right)}
\left( h_{0} - h_{l} \right) < u_{0} 
< -\sqrt{gh_{0},}\;  h_{0} > h_{l}.
\end{equation*}
The fact that shock speed is negative gives another bound on 
\({u}_{0}\)
\[
u_{l} - \sqrt{\frac{g\left( h_{\ast} + h_{l} \right)h_{\ast}}{2h_{l}}}<0 
\text{ that implies }  gh_{\ast}^{2} + gh_{l}h_{\ast} - 2h_{l}u_{l}^{2} > 0.
\]
The roots of that quadratic function are
\[
(h_{\ast})_{1,2} = \frac{- gh_{l} \pm \sqrt{g^{2}h_{l}^{2} + 8gh_{l}u_{l}^{2}}}{2g}.
\]
Thus, 
\[
h_{\ast} > \frac{h_l}{2}\left( \sqrt{1+8z} - 1\right) =: \tilde{h}, \quad \text{where } z:= \frac{u_l^2}{gh_l}.
\]

The value of $h_{\ast}$ is between $h_{l}$ and $h_{0}$, so condition
for its existence is  
\begin{equation*} 
h_{0} > \frac{h_l}{2}\left( \sqrt{1+8z} - 1 \right), \text{ and } h_{0}>h_{l}.
\end{equation*}
The relation $\tilde{h}>h_{l}$ holds for
$z > 1$. If not, the previous bound 
$h_{0}>h_{l}$ is enough for existence of $S_{1}+R_{2}$ solution.

\noindent \textbf{3.}
An $S_{1}$ is much simpler to analyse than 
the previous wave combination. 
The condition for $h_{0}$ is (\ref{SR2}), while $u_{0}$ has to satisfy  
\begin{equation*} 
u_{0} = u_{l} - \sqrt{\frac{g}{2} \left( \frac{1}{h_{0}}+\frac{1}{h_{l}} \right)}
\left( h_{0} - h_{l} \right), \quad h_{0} > h_{l}.
\end{equation*}

\noindent \textbf{4.}
Suppose that there is a single $R_{1}$ solution on the left-hand 
side. That is,  
\(u_{0} = u_{l} + \sqrt{gh_{l}} - \sqrt{gh_{0}}\), \(h_{0}<h_{l}\). 
The condition $\lambda_{1}(h_{0},u_{0})<0$ implies  
\(u_{0} < \sqrt{gh_{0}}\).
From that relations, the condition for \(h_{0}\) given by 
the inequality 
\( u_{l} + \sqrt{gh_{l}} - \sqrt{gh_{0}} < \sqrt{gh_{0}}\). 
It implies \(\sqrt{h_{0}} > \frac{1}{2\sqrt{g}}(u_{l} + \sqrt{gh_{l}})\),
and further on, 
\(h_{0} > \frac{1}{4}\left( \frac{u_{l}}{\sqrt{g}} + \sqrt{h_{l}} \right)^{2} > h_{l}\).
But that contradicts the fact that \(h_{0}<h_{l}\) for $R_1$, so there is no
a single $R_{1}$ solution.

\noindent \textbf{5.}
Finally, suppose that there exists a single $R_{2}$ solution. Then,
\(u_{0} = u_{l} - \sqrt{gh_{l}} + \sqrt{gh_{0}}\), \(h_{0}>h_{l}\),
and \(u_{0} < -\sqrt{gh_{0}}\). The last condition is here because
$\lambda_{2}(h_{0},u_{0})<0$. One could see that $u_{0}>u_{l} \geq 0$ 
since $h_{0}>h_{l}$. But that would mean $u_{0}+\sqrt{gh_{0}}>0$ which
is contradiction with $u_{0}<-\sqrt{gh_{0}}$.

To conclude, the only possible solutions on the left-hand
side of the bed jump (dam) are
\(S_{1}\) with conditions (\ref{S}) and (\ref{SR2}) or 
\(S_{1}+R_{2}\), with satisfied conditions (\ref{SR}) and (\ref{SR2}).
\end{proof}

The next task would be to look at all possibilities for 
the right-hand side waves with positive velocities connecting 
$(h_{1},u_{1})$ with  $(h_{r},u_{r})$ such that conditions
(\ref{chi}) and (\ref{ent}) around the origin
are satisfied.
The analysis may be done similarly to the above but with much more technical
details because $u_{l}\geq 0$ have excluded some cases. 
Instead, we will look at the specific case of a bed jump
being a dam. That is $(h_{r},u_{r})=(0,0)$ 
-- there is no water on the right-hand side of the dam.

\section{Dam breaking problem}

Suppose that there is no fluid on the right-hand side of
the dam initially, i.e.\ $(h_{r},u_{r})=(0,0)$.
The question is whether the fluid will go over it or not, 
i.e.\ is there a solution to (\ref{sys}) that has a non-trivial 
right-going wave? If so, we will see that the speed immediately 
after the jump should be positive, 
$u_{1}>0$. That would imply $u_{0}>0$, which is 
possible only if we have a single shock on the
left-hand side by Lemma \ref{lemma2} (Case 3). 

\begin{lemma}\label{lemma:dam}
Fix $(h_l,u_l),$ $h_l,u_l>0$ and $b_1>b_0>0$. 
If there exists a weak solution to
\[
\begin{split}
& \partial_{t}h + \partial_{x}(uh) = 0 \\
& \partial_{t}(uh) + \partial_{x}\left( u^{2}h + g\frac{h^{2}}{2} \right) = - ghb_{x}, \end{split}\]
with the Riemann initial data
\[
(h,u)|_{t=0} = \begin{cases}
(h_{l},u_{l}), & x < 0 \\
(0,0), & x > 0.\end{cases}
\]
It consists of an $S_1$ wave with a negative speed on the left-hand side,
connected by a centered shadow wave (\ref{sdw}) to a combination of $R_1$ and $R_2$,
with a vacuum state in between. 
\end{lemma}
\noindent
\begin{proof} 
The only way to reach $(0,0)$ from some state $(h_{1},u_{1})$ with both components 
being positive is through an $R_{1}$ followed by the vacuum. That is, find $u_{m}$ such
that the $R_{1}$ connects $(h_{1},u_{1})$ with $(0,u_{m})$, 
$u_{m}=u_{1}+\sqrt{gh_{1}}$,
while $(0,u_{m})$ and $(0,0)$ will be connected by the vacuum. 
That situation is similar to the well-known vacuum solutions 
for the isentropic gas model. The speed $u_{1}$ has to be big enough
because $\lambda_{1}(h_{1},u_{1})=u_{1}-\sqrt{gh_{1}}$ has to be greater
or equal to zero since the wave is going to the right. That implies $u_{0}>0$ because
of (\ref{u1}). We have seen that only $S_1$ or $S_1+R_2$ may be solutions
to the Riemann problem on the left-hand side when $u_l>0$.
The final speed of an $R_{2}$ going from the left-hand side
is negative, $u_{0}+\sqrt{gh_{0}}<0$ which would imply 
$u_{0}<0$. That would mean $u_1<0$, which is in contradiction with the vacuum assumption.
Thus, an $S_{1}$ is only possible solution on the left-hand
side. 
\end{proof}

We reduce the existence problem to a discussion about the possible values of 
the variable $h_{0}$, since all others are determined by it: 
$u_{0}=u_{0}(h_{0})$ by (\ref{S}), $u_{1}=u_{1}(h_{1})$ by (\ref{u1}), 
and 
$h_{1}=h_{1}(h_{0})$ by (\ref{opth1})
(by the maximal dissipation condition and the entropy condition).
The speed of $S_{1}$ defined by 
\begin{equation}\label{u01}
u_0=u_0(h_0)=u_l-\sqrt{\frac{g}{2}\Big(\frac{1}{h_0}+\frac{1}{h_l}\Big)}(h_0-h_l),
\end{equation}
is 
\[
c_{1} = u_{l} - \sqrt{\frac{g}{2}
\left( \frac{1}{h_{0}} + \frac{1}{h_{l}} \right)} h_{0}. 
\]

Using notation and the results of the 
stationary shadow wave admissibility 
in the sense of Definition \ref{d_optim},  we obtain the total energy production of the solution,
\[
\begin{split}
\frac{d}{dt}H_{[-L,L]}(t) = & -c_{1} \left(\eta(h_{0},u_{0}) 
- \eta(h_{l},u_{l})\right)
+ \frac{\varepsilon}{2}\left(\eta(\chi,0) 
-\eta(h_{0},u_{0})\right) \\
& - \frac{\varepsilon}{2}\left( \eta(h_{1},u_{1}) 
-\eta(\chi,0) \right) 
+Q(h_{1},u_{1})-Q(0,u_{m}) \\
\approx &  Q(h_{1},u_{1}) 
-c_{1} \left(\eta(h_{0},u_{0}) - \eta(h_{l},u_{l})\right). 
\end{split}
\]
After letting $\varepsilon \to 0$, the non-constant part of 
$\frac{d}{dt}H_{[-L,L]}(t)$ equals
\begin{equation}\label{E}
\begin{split}
E(h_0):= & -c_{1}
\left( h_{0}\frac{u_{0}^{2}}{2}+g\frac{h_{0}^{2}}{2}+b_{0}h_{0}
-\left( h_{l}\frac{u_{l}^{2}}{2}+g\frac{h_{l}^{2}}{2}+b_{0}h_{l}\right)
\right)\\
& + h_{1}u_{1}\left( \frac{u_{1}^{2}}{2}+g(h_{1}+b_{1})\right),	
\end{split}
\end{equation}
where $c_{1}=c_{1}(h_{0})$, $u_{1}=u_{1}(h_{0})$, and $h_{1}=h_{1}(h_{0})$.
The goal is to find $h_{0}$ that minimizes the functional $E$ and satisfies
the assumptions listed below.
The conditions are:
\begin{enumerate}
\item $h_{0}\geq h_{l}$.
\item Non-positivity of the shock speed,
\begin{equation} \label{C1}
u_{l} - \sqrt{\frac{g}{2} \left( \frac{1}{h_{0}} + \frac{1}{h_{l}} \right)}
h_{0}\leq 0.
\end{equation}
\item Non-negativity of $u_{0}$,  
\begin{equation} \label{C2}
u_{l} - \sqrt{\frac{g}{2} \left( \frac{1}{h_{0}} + \frac{1}{h_{l}} \right)}
\left( h_{0} - h_{l} \right) \geq 0, \quad h_{0} > h_{l}.
\end{equation}
Note that $h_{0}=h_{l}$ gives a constant state on the left-hand side,
and $u_{0}=u_{l}>0$ by the assumption.
\item  The entropy condition
\[h_1\leq h_0-(b_1-b_0).\]
\end{enumerate}
After some algebraic 
manipulations, condition (\ref{C1}) transforms to 
\[
h_{0}>\frac{1}{2}h_{l}\left( \sqrt{1+8\frac{u_{l}^{2}}{gh_{l}}}-1 \right).
\]
Denote  $z=\frac{u_{l}^{2}}{gh_{l}}$. One can easily see that 
\[
\frac{1}{2}h_{l}\left( \sqrt{1+8z}-1 \right)>h_l\quad \text{ for }
z>1.
\]
Thus, (\ref{C1}) implies
\[
h_{0}\geq \tilde{h}=
\begin{cases}
\frac{1}{2}h_{l}\left( \sqrt{1+8z}-1 \right), & 
z>1 \\
h_{l}, & z\leq 1
\end{cases}.
\]
Note that $c_1(\tilde{h})=0$ when $z>1,$ while $c_1(\tilde{h})=c_1(h_l)\leq 0$ if $z\leq 1.$

Relation (\ref{C2}) is true if and only if 
\begin{equation} \label{f}
f(h_{0}):=h_{0}^{3}-h_{l}h_{0}^{2}
-(h_{l}^{2}+\frac{2}{g}u_{l}^{2}h_{l})h_{0}+h_{l}^{3}<0.
\end{equation}
We have $f(0)=h_{l}^{3}>0$ and $f(h_{l})=-\frac{2}{g}u_{l}^{2}h_{l}^{2}<0$,
so $f$ has three real roots, one less than zero, one between $0$ and
$h_{l}$, and one greater than $h_{l}$.
The extreme points of the cubic function $f$ are 
\[
(h_{0})_{e_{1},e_{2}}=\frac{1\pm 2\sqrt{1+\frac{3}{2}z}}
{3}h_{l}.
\]
The first extreme point is negative, while the second one is greater than $h_{l}$.
Denote by $\underline{h}$ the largest real root of the equation $f(h_0)=0$.
After that point, $f$ is positive.

For $z\leq 1, $ $\tilde{h}=h_l<\underline{h}$, and
\[\begin{split}
f(\tilde{h})&=h_l^3\Big(\frac{1}{8}(\sqrt{1+8z}-1)^3-\frac{1}{4} (\sqrt{1+8z}-1)^2
-\frac{1+2z}{2}(\sqrt{1+8z}-1)+1\Big)\\
&=h_l^3\Big(1-\frac{1}{2}\sqrt{1+8z}(\sqrt{1+8z}-1)\Big)<-2h_l^3<0, \,\text{ for } z>1.
\end{split}\] 
So $f(\tilde{h})<0$ for each $z>0$, which means that $\tilde{h}<\underline{h}.$
Hence, the conditions 1-3 reduce to $h_0\in[\tilde{h},\underline{h}].$
If $h_0=\underline{h}$, then $u_0=u_1=h_1=0$. 
\medskip

Before we prove the existence theorem, we shall prove the following auxiliary lemma.
\begin{lemma}\label{lemma_H}
Fix $h_l,u_l>0$, $b_0>0$ and define
\[\tilde{H}(h):=-c_1(h)\big(\eta(h,u_0(h))-\eta(h_l,u_l)\big), \quad h\geq h_l,\]
where $u_0(h)=u_{l} - \sqrt{\frac{g}{2} \left( \frac{1}{h} + \frac{1}{h_{l}} \right)}
\left( h - h_{l} \right),$
$c_1(h)= u_{l} - \sqrt{\frac{g}{2} \left( \frac{1}{h} + \frac{1}{h_{l}} \right)}
h$ and $\eta(h,u)=h\frac{u^2}{2}+g\frac{h_0^2}{2}+b_0h$.
Then $\tilde{H}$ is non-negative and increasing function for $h\in(\tilde{h},\underline{h}]$.
\end{lemma}
\begin{proof}
Denote
\[d(h):=\eta(h,u_0(h))-\eta(h_l,u_l).\]
Then
\[\begin{split}
\tilde{H}'(h)&=-c_1'(h)d(h)-c_1(h)d'(h),\\
c_1'(h)&=\sqrt{\frac{g}{2}\Big(\frac{1}{h}+\frac{1}{h_l}\Big)}
\Big(-1+\frac{1}{2}\frac{h_l}{h+h_l}\Big)<0
\end{split}\]
and
\[
d'(h)=\frac{u_0^2(h)}{2}+h(u_0(h)u_0'(h)+g)+b_0.
\]
We have 
\[d(h)=(h-h_l)\Big(b_0+\frac{1}{2}c_1^2(h)+\frac{hh_l}{2}\cdot\frac{g}{2}
\Big(\frac{1}{h}+\frac{1}{h_l}\Big)\Big)>0, \quad h>h_l,\]
which implies $\tilde{H}(h)> 0$ for $h\in(\tilde{h},\underline{h}]$, 
and $\tilde{H}(\tilde{h})=0$. Using  
\[-\frac{1}{2}h_l-\frac{h^2}{h+h_l}\geq -h, \quad h\geq h_l,\]
one easily proves 
\[h(u_0(h)u_0'(h)+g)\geq -hc_1(h)\sqrt{\frac{g}{2}
\Big(\frac{1}{h}+\frac{1}{h_l}\Big)}+\frac{g}{2}\Big(\frac{1}{h}+\frac{1}{h_l}\Big)
\cdot \frac{1}{2}h_l(h-h_l)> 0, h>h_l.\]
Thus, $d'(h)>0$ for $h>h_l$. This is sufficient to prove that 
$\tilde{H}'(h)>0$ for $h\in(\tilde{h},\underline{h}].$
\end{proof}

\begin{theorem}
Let assumptions of Lemma \ref{lemma:dam} hold 
and suppose $\underline{h}\geq [b]$.
Then there exist solution from Lemma \ref{lemma:dam}
satisfying the maximum dissipation principle and the entropy condition. 
Let
$\hat{h}$ be the solution of the equation $r(h_0)=0,$ where
\[r(h_0):=\sqrt[3]{\frac{h_0^2u_0^2}{g}}-h_0+[b].\]
Let
\[M_1=\min_{h_0\in[\tilde{h},\underline{h}]}\Big\{E(h_0)\Big|
\,h_1=\sqrt[3]{\frac{h_0^2u_0^2}{g}},\,r(h_0)\leq 0\Big\}=E(h_m).\]
Denote 
\[\mathcal{M}_2:=\Big\{h_0\in[\tilde{h},\hat{h}]\Big|\,
\frac{[b]}{h_0}\leq\frac{3-\sqrt{1+8y}}{2},\,y=\frac{u_0^2}{gh_0}\leq 1\Big\}\]
and if $\mathcal{M}_2$ is nonempty, 
define
\[M_2=\min_{h_0\in \mathcal{M}_2}\Big\{E(h_0)\Big|\,h_1=h_0-[b]\Big\}=E(h_n).\]
The connection $\big((h_0,u_0),(h_1,u_1)\big)$ that satisfies the entropy condition
and minimizes the total entropy production is given by
\[(h_0,h_1)=\begin{cases}
\big(h_n,h_n-[b]\big), &\text{ if } \mathcal{M}_2 \text{ is nonempty and }  M_2<M_1\\
\big(h_m,\sqrt[3]{\frac{h_m^2u_0^2(h_m)}{g}}\big), & \text{ otherwise}
\end{cases}\] 
while $u_0$ is given by (\ref{u01}) and $u_1=\frac{h_0u_0}{h_1}$. 
\end{theorem}

\begin{proof}
The above analysis shows that  $h_0\in[\tilde{h},\underline{h}].$ An optimal
$h_0\in[\tilde{h},\underline{h}]$ has to be chosen such that the entropy condition (\ref{ent}) holds, 
while total energy production is minimal.  
Lemma \ref{lemma1} and (\ref{opth1}) suggest two possibilities for $h_1(h_0)$. 
Let us check the first one, $h_1(h_0)=\bar{h}_1(h_0)=\sqrt[3]{\frac{h_0^2u_0^2}{g}}$.

Entropy condition for  $h_1=\sqrt[3]{\frac{h_0^2u_0^2}{g}}$ from (\ref{opth1}) 
is satisfied if $r(h_0)\leq 0$ for some $h_0\in[\tilde{h},\underline{h}]$.
The function $r(h_0)$ is decreasing in the interval $[\tilde{h},\underline{h}]$ which follows from
\[\begin{split}
u_0'(h_0)&=-\frac{g}{2}\frac{1}{h_0}\frac{\frac{1}{2}h_l
\big(\frac{1}{h_0}+\frac{1}{h_l}\big)+\frac{h_0}{h_l}}{\sqrt{\frac{g}{2}
\big(\frac{1}{h_0}+\frac{1}{h_l}\big)}}<0,\\
(h_0u_0)'&=u_l+\sqrt{\frac{g}{2}\Big(\frac{1}{h_0}+\frac{1}{h_l}\Big)}
\Big(-h_0+\underbrace{\frac{1}{2}h_l-\frac{h_0}{h_l}
\frac{h_0h_l}{h_0+h_l}}_{<\frac{1}{2}(h_l-h_0)<0}\Big)<c_1<0
\end{split}\]
and 
\[r'(h_0)=\frac{2}{3}\frac{1}{\sqrt[3]{g}}(h_0u_0)^{-\frac{1}{3}}(h_0u_0)'-1<0.\]

Due to assumption $[b]\leq \underline{h}$, we have $r(\underline{h})=[b]-\underline{h}\leq 0$,
so the entropy condition is satisfied at least for some $h_0\in[\tilde{h},\underline{h}]$
and $h_m$ always exists. If $[b]=\underline{h}$, we have $h_m=\underline{h}$
and $u_0(h_m)=u_1(h_m)=0.$ Otherwise, $h_m<\underline{h}$, since $u_1(\underline{h})=0$ and 
\[\frac{\partial E(\underline{h})}{\partial h_0}
=\frac{\partial \tilde{H}(\underline{h})}{\partial h_0}>0,\]
see Lemma \ref{lemma_H}.

If  $r(\tilde{h})\leq 0,$ then $r(h_0)\leq 0$ for each $h_0\in[\tilde{h},\underline{h}]$, 
meaning that the entropy condition holds for each $h_0$ in that interval.
The admissible connection is obtained by choosing $h_0=h_m$ and 
$\hat{h}_1=\sqrt[3]{\frac{h_m^2u_0^2(h_m)}{g}}$. 
Note that the solution $h_0=h_m$ is uniquely determined by
$h_1=\bar{h}_1(h_0)=\sqrt[3]{\frac{h_0^2u_0^2(h_0)}{g}}$,
since $h_0u_0$ is decreasing with respect to $h_0$.

Finally, $r(\tilde{h})>0$ means that there exists $\hat{h}\in(\tilde{h},\underline{h}]$ such that
$r(\hat{h})=0$, i.e.  $\sqrt[3]{\frac{\hat{h}^2u_0(\hat{h})^2}{g}}=\hat{h}-[b]$.
Since the function $r$ is decreasing, the entropy condition holds
for $h_0\in[\hat{h},\underline{h}]$. 
\noindent
Thus, $M_1$ is defined as minimum of the function $E$ on the interval
$[\max\{\tilde{h},\hat{h}\},\underline{h}]$ with $h_1=\hat{h}_1$
and $\hat{h}_1=\sqrt[3]{\frac{h_m^2u_0^2(h_m)}{g}}$.

The second possibility from Lemma \ref{lemma1} is to take $h_1=h_0-[b]$ 
if $\sqrt[3]{\frac{h_0^2u_0^2}{g}}>h_0-[b]$ for some $h_0.$
If $\hat{h}\leq\tilde{h}$, $h_0=h_m$ gives the optimal solution.
Otherwise, we have to check if it is possible that some $h_0\in[\tilde{h},\hat{h}]$
with $h_1=h_0-[b]$ will give total energy production smaller than $M_1$.

\begin{comment}
Having in mind that not every $h_0\in [\tilde{h},\hat{h})$  will give positive $h_1=h_0-[b],$
we have to find condition which will provide that $\chi$ (i.e $h_1$) is positive.
Replacing $h_1=h_0-[b]$ in (\ref{h1}), one gets
\[(h_0-[b])^3-(h_0^2+\frac{2}{g}h_0u_0^2-2[b]\chi)(h_0-[b])+\frac{2h_0^2u_0^2}{g}=0.\]
That implies
\[\chi=\frac{1}{2[b]}\Big(-\frac{[b]}{h_0-[b]}+h_0[b]+(h_0-[b])[b]\Big)\]
and $\chi>0$ if
\[g(h_0-[b])^2+h_0g(h_0-[b])-2h_0u_0^2>0.\]
Denoting $w := h_0-[b]$, the above relation can be written as
\[w^2+h_0w-2h_0^2y>0, \quad y=\frac{u_0^2}{gh_0}.\]
That is,
\[w=h_0-[b]>\frac{h_0}{2}(\sqrt{1+8y}-1).\]
That relation is possible only if $y<1,$ since for $y\geq 0$, the right-hand side is greater than
$h_0$, and there is no solution for connecting SDW in that case.
\end{comment}
So, let us find the minimum of $E$ when $h_1=h_0-[b]$ for $h_0\in[\tilde{h},\hat{h}]$. 
The admissible connecting SDW exists if
\[1-\sqrt[3]{y}\leq \frac{[b]}{h_0}<\frac{3-\sqrt{1+8y}}{2} \quad
\text{and} \quad y=\frac{u_0^2}{gh_0}\leq 1.\]
Note that $\frac{u_0^2}{gh_0}$ is decreasing with respect to $h_0$,
so the first inequality above does not hold if $h_0> \hat{h}$. 
If it exists, $h_n\in[\tilde{h},\hat{h}]$ and $h_n\leq h_m.$ 

If both $M_1$ and $M_2$ exist, the smaller one will determine the optimal $h_0$ and $h_1.$ 

It is easy to prove that $Q\big(h_0-[b],\frac{h_0u_0}{h_0-[b]}\big)
>Q\big(\bar{h}_1(h_0),\frac{h_0u_0}{\bar{h}_1(h_0)}\big)$, 
so if $h_n=\hat{h}$, then $M_2>E(\hat{h})\Big|_{h_1=\bar{h}_1(\hat{h})}\geq M_1$. 
All numerical experiments indicate that $M_1 < M_2$; however, we were unable to prove this. 
Thus, the question of the uniqueness of the solution is still open since 
we were not able to exclude the possibility that   $M_1=M_2.$
\end{proof}

As we can see above, if $\underline{h}=[b]$, the fluid will not pass the dam since the optimal solution is $h_0=\underline{h}$ and $u_0=u_1=h_1=0$. If $\underline{h}<[b]$, then $h_0-[b]<0$ for each $h_0\in[\tilde{h},\underline{h}]$.
That implies $h_1>h_0-[b]$, since $h_1$ has to be non-negative. Thus, the entropy condition is not satisfied for $h_0\in[\tilde{h},\underline{h}]$.
\begin{corollary}
Suppose $\underline{h}<[b]$. Then, the admissible solution does not exist, i.e.,  the fluid will not pass the dam.
\end{corollary}

\section*{Conclusion}

This paper presents a solution to the shallow water system with a bottom containing a single jump.
Mathematically, the problem can be viewed as a system of conservation laws with
a discontinuous flux function or as a balance law with a measure as the source function.
We have chosen the latter approach.

Physical energy plays a role analogous to mathematical entropy.
With this choice, the maximum dissipation principle has proven to be effective
in selecting an appropriate approximate solution to the problem.

This solution is a net of piecewise differentiable functions that converge to
a locally integrable function. It can be described as a singular shock solution with zero mass,
as there are no singular measures in the limit.

We focus on the dam-break Riemann problem, where the initial data are Riemann-type with
a vacuum on the right-hand side. While the approach used in this paper appears sufficient
for an arbitrary Riemann problem, the number of cases would be too large to be practical.
Therefore, we have chosen to present our idea using a simpler case.

A natural next step in investigating this system would be to introduce
a vanishing viscosity term and analyze the limit. 

\section*{Acknowledgment}
This work was supported by 
Science Fund of the Republic of Serbia, GRANT No TF C1389-YF (“FluidVarVisc”).

\noindent
Dalal Daw 

\noindent
University of Zawia, Faculty of Science (Ajelat), Libya 

\noindent
Email: d.daw@zu.edu.ly
\medskip

\noindent
Marko Nedeljkov 

\noindent
Department of Mathematics and Informatics, Faculty of Sciences, 
University of Novi Sad 

\noindent
Email: marko@dmi.uns.ac.rs

\medskip

\noindent
Sanja Ru\v zi\v ci\' c 

\noindent
Department of Mathematics and Informatics, Faculty of Sciences, 
University of Novi Sad 

\noindent
Email: sanja.ruzicic@dmi.uns.ac.rs

\end{document}